\newif\ifpictures
\picturestrue

\documentclass[12pt]{amsart}
\usepackage{amsrefs}
\usepackage{amsfonts}
\usepackage{amsmath}
\usepackage{amssymb}
\usepackage{amsthm}
\usepackage {color, tikz}
\usepackage{thmtools}
\usepackage{thm-restate}
\usepackage{fullpage,url,amssymb,amsmath,graphicx, amsthm, latexsym, mathrsfs, url}
\usepackage[colorlinks=true,linkcolor=blue,urlcolor=blue, citecolor=blue]%
  {hyperref}
\usepackage[T1]{fontenc}
\usepackage{mathptmx}
\usepackage{microtype}
\usepackage {enumerate,booktabs}
\usepackage {color, tikz}
\usepackage{graphics}
\usepackage[all]{xy}

\numberwithin{equation}{section}
\newtheorem{thm}{Theorem}
\newtheorem*{thm*}{Theorem}

\newtheorem{lemma}[thm]{Lemma}

\newtheorem{conj}[thm]{Conjecture}
\newtheorem{definition}[thm]{Definition}
\newtheorem*{question}{Open Problem}
\newtheorem*{definition*}{Definition}
\newtheorem*{cor*}{Corollary}

\newtheorem{remark}[thm]{Remark}

\numberwithin{thm}{section}
\numberwithin{definition}{section}
\numberwithin{lemma}{section}
\numberwithin{cor}{section}
\numberwithin{example}{section}
\numberwithin{prop}{section}
\numberwithin{conj}{section}
\numberwithin{remark}{section}

\DeclareMathOperator{\codim}{codim}

\numberwithin{equation}{section}

\newcommand{\Osh}{{\mathcal O}}

\newcommand{\Sym}{\operatorname{Sym}}

\newcommand{\PP}{\mathbb{P}}

\newcommand{\CC}{\mathbb{C}}
\newcommand{\RR}{\mathbb{R}}

\newcommand{\sI}{\ensuremath{\kern -1pt \mathscr{I}\kern -2pt}} 
\newcommand{\sJ}{\ensuremath{\kern -2pt \mathscr{J}\kern -2pt}}

\begin{document}

\title[]{Gap vectors of real projective varieties.}

\author[G.~Blekherman]{Grigoriy Blekherman}
\address{Greg Blekherman\\ School of Mathematics \\ Georgia Tech, 686 Cherry
  Street\\ Atlanta\\ GA\\ 30332\\ USA}
\email{\href{mailto:greg@math.gatech.edu}{greg@math.gatech.edu}}

\author[S.~Ilman]{Sadik Iliman}
\address{Sadik Iliman\\ FB 12 - Institut f\"ur Mathematik\\ Goethe-Universit\"at\\ Robert-Mayer-Strasse 10\\ 60325 Frankfurt am Main}
\email{\href{mailto:iliman@math.uni-frankfurt.de}{iliman@math.uni-frankfurt.de}}

\author[M.~Juhnke-Kubitzke]{Martina Juhnke-Kubitzke}
\address{Martina Juhnke-Kubitzke \\ FB 12 - Institut f\"ur Mathematik\\ Goethe-Universit\"at\\ Robert-Mayer-Str. 10\\60325 Frankfurt am Main}
\email{\href{mailto: kubitzke@math.uni-frankfurt.de}{kubitzke@math.uni-frankfurt.de}}

\author[M.~Velasco]{Mauricio Velasco} \address{Mauricio Velasco\\ Departamento
  de Matem\'aticas\\ Universidad de los Andes\\ Carrera 1 No. 18a 10\\ Edificio
  H\\ Bogot\'a\\ Colombia}
\email{\href{mailto:mvelasco@uniandes.edu.co}{mvelasco@uniandes.edu.co}}

\subjclass[2010]{14P99, 14M99, 52A99}

\begin{abstract} Let $X\subseteq \PP^m$ be a totally real, non-degenerate, projective variety and let $\Gamma\subseteq X(\RR)$ be a generic set of points. Let $P$ be the cone of nonnegative quadratic forms on $X$ and let $\Sigma$ be the cone of sums of squares of linear forms. We examine the dimensions of the faces $P(\Gamma)$ and $\Sigma(\Gamma)$ consisting of forms in $P$ and $\Sigma$, which vanish on $\Gamma$. As the cardinality of the set $\Gamma$ varies in $1,\dots,\codim(X)$, the difference between the dimensions of $P(\Gamma)$ and $\Sigma(\Gamma)$ defines a numerical invariant of $X$, which we call the gap vector of X. In this article we begin a systematic study of its fundamental properties. Our main result is a formula relating the components of the gap vector of $X$ and the quadratic deficiencies of $X$ and its generic projections. The quadratic deficiency is a fundamental numerical invariant of projective varieties introduced by F. L. Zak. The relationship between quadratic deficiency and gap vectors allows us to effectively compute the gap vectors of concrete varieties as well as to prove several general properties. We prove that gap vectors are weakly increasing, obtain upper bounds for their rate of growth and prove that these upper bounds are eventually achieved for all varieties. Moreover, we give a characterization of the varieties with the simplest gap vectors: We prove that the gap vector vanishes identically precisely for varieties of minimal degree, and characterize the varieties whose gap vector equals $(0,\dots, 0,1)$. In particular, our results give a new proof of the theorem of Blekherman, Smith and Velasco saying that there are nonnegative quadratic forms which are not sums of squares on every variety, which is not of minimal degree. Finally, we determine the gap vector of all Veronese embeddings of $\PP^2$, generalizing work of the first three authors.\end{abstract}

\maketitle

\section{Introduction}

Let $X\subseteq \PP^m$ be a totally real, non-degenerate, projective variety with homogeneous coordinate ring $R$. The vector space $R_2$ of real quadratic forms on $X$ contains two natural convex cones: the cone $P$ of nonnegative quadratic forms and the cone $\Sigma$ of sums of squares of linear forms. The relationship between these two cones has been an object of much interest since the time of Hilbert~\cite{Hilbert}. Recent results (see~\cite{Blekherman} and \cite{BSV}) suggest that this relationship is in fact controlled by algebro-geometric features of the variety $X$ over $\CC$. A more complete understanding of the relationship between the algebraic geometry of $X$ and the convex geometry of the cones $\Sigma$ and $P$ is, in our view, an intrinsically interesting problem with the potential for bringing new insights and tools from algebraic geometry to optimization. The study of this relationship is arguably the central problem of the emerging field of {\it Convex Algebraic Geometry}~\cite{BPT}. The purpose of this article is to contribute to this line of research by asking: What geometric features of $X$ control the dimensions of ``generic'' exposed faces of $P$ and $\Sigma$? 
To more precisely formulate this question we introduce the gap vector, a new numerical invariant of real projective varieties,
\begin{definition}  For a set $\Gamma\subseteq X(\RR)$, let $\Sigma(\Gamma)$ and $P(\Gamma)$ consist of forms in $\Sigma$ and $P$, which vanish on $\Gamma$. For $1\leq j\leq c:=\codim(X)$, we define
$g_j(X):=\dim (P(\Gamma))-\dim(\Sigma(\Gamma))$, where $\Gamma\subseteq X(\RR)$ is a generic set of points of cardinality $j$. These numbers are arranged in the gap vector of $X$, denoted by
$g(X):=(g_1(X),\dots, g_c(X))$.\end{definition}

The purpose of the article is to characterize the fundamental properties of the gap vector. We demonstrate that this invariant is determined by the geometry of the variety $X$ over $\mathbb{C}$ and that it has strong connections to classical topics in algebraic geometry, such as the dimensions of secant varieties, the geometry of generic projections and the quadratic deficiency of a variety. 

A non-vanishing gap vector gives a certificate for strict inclusion of $\Sigma$ in $P$. Using dimensional differences to certify this strict inclusion has a long history starting with Hilbert's original paper \cite{Hilbert} on globally nonnegative polynomials, which launched the investigation of connections between nonnegativity and sums of squares. The dimensional difference aspect of Hilbert's approach was made clear by a modern exposition of Hilbert's method due to Reznick \cite{Rez}. We note that globally nonnegative forms of degree $2d$ correspond to the special case where $X=\nu_d(\PP^n)$ is the $d$\textsuperscript{th} Veronese embedding of $\PP^n$. For this situation dimensional differences between $P$ and $\Sigma$ were studied by the first three authors in \cite{BIK}. One of the objectives of the present article is to refine and generalize these results to projective varieties. 

There are several advantages of working with arbitrary projective varieties. First, to understand nonnegative forms and sums of squares of degree $2d$ on a variety $X$, it suffices to study the cones $P$ and $\Sigma$ on the variety $Y=\nu_d(X)$. Second, the setting of projective varieties and gap vectors gives us both a language to describe interesting properties of the relationship between nonnegative polynomials and sums of squares, and a context in which to look for answers. For instance, in this paper we answer the following questions:
\begin{itemize}
\item{For which varieties do we have a generic dimensional certificate that $\Sigma\neq P$? For which varieties does the gap vector have non-zero entries?}
\item{For which varieties do we have $\Sigma\neq P$ and this difference is as small as possible? For which varieties is the gap vector non-vanishing and of the form $(0,\dots,0,1)$?}
\item{How do the dimensions of generic faces $\Sigma(\Gamma)$ and $P(\Gamma)$ change as $|\Gamma|$ increases? Is the gap vector monotonic and if so, what is its growth rate?}
\end{itemize}

The answer to the questions above is a first step towards understanding the general properties of gap vectors. In our opinion a complete determination of this invariant on real projective varieties is a problem of much interest, which, we suspect, would also bring new connections between convex and classical algebraic geometry.

In the remainder of this section, we describe the organization and the main results of the article. In Section~$\S\ref{FacesP}$ we generalize from \cite{BIK} a geometric property of nonempty finite sets $\Gamma\subseteq X$, called {\it independence}.

\begin{definition*} Let $\Gamma\subseteq X\subseteq \PP^m$ be a finite set of nonsingular points and let  $\langle \Gamma\rangle$ be the projective subspace of $\PP^m$ spanned by $\Gamma$. The set $\Gamma$ is independent if
\begin{enumerate}
\item{the (set-theoretic) equality $\langle \Gamma\rangle \cap X=\Gamma$ holds, and}
\item{ for every point $p\in \Gamma$, the vector space $T_pX\cap \langle \Gamma\rangle$ is trivial, where $T_p$ is the tangent space at $p$.}
\end{enumerate}
\end{definition*}

Independent sets of points are useful for two reasons. First, Theorem~\ref{thm:DimP} shows that if $\Gamma\subseteq X(\RR)$ is an independent set of points, then $P(\Gamma)$ is full-dimensional in the vector space of quadratic forms on $X$, which vanish at the points of $\Gamma$ with multiplicity two. If $\Gamma$ is generic, then the dimension of this vector space, and hence the dimension of $P(\Gamma)$, is determined by that of the $(|\Gamma|-1)$\textsuperscript{st} secant variety of $\nu_2(X)$. Second, independent sets are abundant in the sense that a {\it generic} set of points of size at most $\codim(X)$ satisfies the independence property (see  Lemma~\ref{MaxGI}). These two aspects of independence enable us to compute the dimension of $P(\Gamma)$. More precisely, the following theorem holds,
\begin{restatable}{thm}{DimP}\label{thm:DimP}  Let $X\subseteq \PP^m$ be a non-degenerate and totally real variety of dimension $d$. 
 Let $\Gamma\subseteq X(\RR)$ be a generic and independent set of points of size $s$. The following equality holds, 
\[\dim(P(\Gamma))= \dim(R_2)-s(\dim(X)+1).\]
\end{restatable}
Next, we study the dimension of the cone $\Sigma(\Gamma)$ and prove that it is determined by the image of $X$ via the projection away from $\Gamma$. More specifically, recall that if $V\subseteq R_1$, then there is a rational map $\phi: \PP^m\dasharrow \PP(V^*)$, which sends a point $p\in \PP^m$ to the projectivization of the linear functional $\ell \in V^*$ defined by $\ell(f)=f(\bar{p})$, where $\bar{p}$ is an affine representative of $p$. In case $V$ is the vanishing linear series of a set of points $\Gamma \subseteq X$ the map $\phi$ is the projection away from $\Gamma$, denoted by $\pi_{\Gamma}$. We have the following theorem:

\begin{restatable}{thm}{DimGamma} \label{thm:DimS} Let $\Gamma\subseteq X(\RR)$ be a finite set of points and let $V\subseteq R_1$ be the linear series consisting of forms vanishing at all points of $\Gamma$. Let $\pi_{\Gamma}: \PP^m\dasharrow \PP(V^*)$ be the rational map determined by $V$ , let $Y$ be the image of $X$ under $\pi_{\Gamma}$ and let $S$ be the homogeneous coordinate ring of $Y$. Then we have $\dim\Sigma(\Gamma)=\dim(S_2)$.
\end{restatable}
In Section~$\S\ref{DGepsilon}$, we combine Theorem \ref{thm:DimP} and Theorem \ref{thm:DimS} to obtain our main result, which is a concrete formula for the entries of gap vectors of projective varieties. Crucially, we are able to express the entries of $g(X)$ just in terms of quadratic deficiencies. Recall that the quadratic deficiency is a key invariant of projective varieties defined by F. L. Zak in~\cite{Zak} in the following way.

\begin{definition*} Let $X\subseteq \PP^m$ be a non-degenerate variety of dimension $d$ and codimension $c$. Let $I$ be the ideal of definition of $X$, and let $R$ be its homogeneous coordinate ring. The quadratic deficiency of $X$ is the number
\[ \epsilon(X):= \binom{c+1}{2}-\dim(I_2)=\dim(R_2)-(m+1)(d+1)+\binom{d+1}{2}.\]
\end{definition*}

\begin{restatable}{thm}{DimGap} \label{DimGap}Let $\Gamma\subseteq X(\RR)$ be a generic set of points of cardinality $j\leq \codim(X)$ and let $Y$ be the image of $X$ under projection away from $\Gamma$. The dimensional gap equals
\[g_j(X)=\dim(P(\Gamma))-\dim(\Sigma(\Gamma))= \epsilon(X)-\epsilon(Y).\]
\end{restatable}

This theorem has several consequences. It provides us with an algorithmic method to compute gap vectors, and allows us to use known constraints on quadratic deficiencies to derive several fundamental restrictions on the combinatorics of gap vectors. These topics are explored in Section~$\S\ref{GapVectors}$. To describe such restrictions, recall that any non-degenerate variety $X\subseteq \PP^m$ satisfies  $\deg(X)\geq \codim(X)+1$ and that $X$ is called of minimal degree if equality is achieved. In Section~$\S\ref{GapVectors}$ we prove the following facts about gap vectors:

\begin{restatable}{thm}{PropGap} \label{gapVector}
Let $c:=\codim(X)$. The  gap vector $g(X)$ has the following properties,
\begin{enumerate}
\item{ $g(X)$ has nonnegative entries and is weakly increasing.}
\item{ $g_{c}(X)=\epsilon(X)$.}
\item{$g_{c-1}(X)=\begin{cases} 0, &\mbox{ if } X \mbox{ is a variety of minimal degree}\\ \epsilon(X)-1,&\mbox{ otherwise}\end{cases}$}
\item{ The following statements hold for $1\leq j\leq c-1$:
\begin{enumerate}
\item{ $g_{j+1}(X)-g_j(X)\leq c-j$}
\item{ Equality in (a) holds if and only if there are no quadrics in the ideal of the variety $Y_j$, obtained by projecting $X$ away from a generic subset $\Gamma\subseteq X(\RR)$ of cardinality $j$.}
\item{ If for some index $s$ we have $g_{s+1}(X)-g_s(X)=c-s$, then $g_{j+1}(X)-g_j(X)= c-j$ for all $j\geq s$.}
\end{enumerate}
}
\end{enumerate}
\end{restatable}
Moreover, we are able to characterize those varieties whose gap vector is as simple as possible.
\begin{restatable}{cor}{PropGapVect}
\label{cor:GapProp}
 The following statements hold,
\begin{enumerate}
\item{ $g(X)\equiv 0$ if and only if  $X$ is a variety of minimal degree.} 
\item{ $g(X)$ has only one nonzero component if and only if $\epsilon(X)=1$ and in that case $g(X)=(0,\dots,0,1)$.}
\end{enumerate}
\end{restatable}

In particular part $(1)$ of the corollary gives a new proof of the theorem of Blekherman, Smith and Velasco~\cite{BSV}, which says that $P\neq\Sigma$ whenever $X$ is not a variety of minimal degree. Part $(2)$ of the corollary characterizes the varieties for which $\Sigma$ and $P$ are closest in terms of gap vector without being equal. 
Finally, in Theorem~\ref{ThmPP2} we explicitly compute the gap vectors of the Veronese embeddings of $\PP^2$ and thus determine the dimensional differences between nonnegative polynomials of degree $2d$ and sums of squares of degree $d$ in three variables vanishing at generic sets of points of size $j$ with $1\leq j\leq \binom{d+2}{2}-3$. Additionally, the case of Veronese embeddings of $\PP^2$ shows that gap vectors can indeed exhibit maximal growth. We also conjecture an explicit formula for gap vectors of Veronese embeddings of $\PP^n$ in Conjecture \ref{conjecture}. 

We conclude with an open problem, whose solution would not only shed further light on the relationship between nonnegative polynomials and sums of squares, but, we suspect, would also bring new connections to classical algebraic geometry.
\begin{question} Describe the gap vectors of non-degenerate, totally real projective varieties. In particular, determine which geometric invariants characterize the smallest index for which the gap vector is non-zero and classify the varieties $X$ with gap vector $(0,\dots,0,1,2).$
\end{question}
\subsection*{Acknowledgements}
We wish to thank the organizing committee of the SIAM Algebraic Geometry Group meeting 2013, which gave us the occasion to start this project. We wish to thank Gregory G. Smith for several helpful conversations during the completion of this project. M. Velasco was partially supported by the FAPA funds from Universidad de los Andes. G. Blekherman was partially supported by an Alfred P. Sloan Research Fellowship and NSF CAREER Award.

\section{Dimensions of faces of the cones $P$ and $\Sigma$}\label{FacesP}
Let $X\subseteq \PP^m$ be a non-degenerate, totally real algebraic variety (i.e., $X$ is a geometrically integral scheme of finite type over ${\rm Spec}(\RR)$ such that the set of real points $X(\RR)$ is Zariski dense in $X$ with a fixed closed immersion $i:X\rightarrow \PP^m_{\RR}$). The variety $X$ is defined by a real-radical homogeneous prime ideal $I\subseteq \RR[x_0,\dots, x_m]$, which contains no linear forms, and we let $R:=\RR[x_0,\dots, x_m]/I$ be the homogeneous coordinate ring of $X$. If $p\in X(\RR)$ and $s\in R_{2d}$ is a form of even degree, which does not vanish at $p$, then $s$ has a well defined sign, either positive or negative, at $p$. The form $s$ is nonnegative on $X$ if at every point $p\in X(\RR)$, $s$ either vanishes or has positive sign. For each $p\in X(\RR)$, we fix a representative $\ell_p\in R_2^*$ obtained by evaluating forms in $R_2$ at a fixed affine representative of $p$. For a set $\Gamma\subseteq \PP^m$, we let $\langle\Gamma\rangle\subseteq \PP^m$ be the projective subspace spanned by $\Gamma$.

\begin{definition} Let $P\subseteq R_2$ be the set of quadratic forms which are nonnegative on $X$ and let $\Sigma\subseteq R_2$ be the set of sums of squares of forms in $R_1$. If $\Gamma\subseteq X(\RR)$ is a finite set of nonsingular points, let $P(\Gamma)$ and $\Sigma(\Gamma)$ be the subsets of $P$ and $\Sigma$ consisting of forms vanishing at all points of $\Gamma$.\end{definition}

\begin{remark} Since $X$ is totally real, the convex cones $P$ and $\Sigma$ are proper cones (i.e., full-dimensional, closed and pointed subsets of $R_2$). The sets $P(\Gamma)$ and $\Sigma(\Gamma)$ are exposed faces of $P$ and $\Sigma$ and, in particular, proper cones themselves, since they lie in the hyperplane $\sum_{p\in \Gamma}\ell_p=0$, which is supporting to both $P$ and $\Sigma$. 
\end{remark}

In this section we study the dimension of the faces $P(\Gamma)$ and $\Sigma(\Gamma)$ reducing their computation to that of some explicit vector spaces. We first focus on the dimension of the cone $\Sigma(\Gamma)$ and prove Theorem \ref{thm:DimS}.
\begin{proof}[Proof of Theorem \ref{thm:DimS}] Let $\pi_{\Gamma}^*: S\rightarrow R$ be the (injective) pull-back map on homogeneous coordinate rings. If $q\in R_2$ is a sum of squares vanishing at the points of $\Gamma$, then each of the linear forms in its summands must vanish at all points of $\Gamma$. It follows that $q=\phi^{*}(a)$ for some sum of squares $a\in S_2$ and that $\Sigma(\Gamma)$ can be identified with the cone of sums of squares in $S_2$. Since the second Veronese variety of a projective space is non-degenerate, the squares in $S_2$ span $S_2$ as a vector space and thus the cone $\Sigma(\Gamma)$ is full-dimensional in $S_2$ as claimed. 
\end{proof}

Next, we focus on the dimension of the faces $P(\Gamma)$ for sets of points $\Gamma$ satisfying the following additional key property:

\begin{definition} Let $\Gamma\subseteq X$ be a finite set of nonsingular points. $\Gamma$ is independent if
\begin{enumerate}
\item{ the (set-theoretic) equality $\langle \Gamma\rangle \cap X=\Gamma$ holds, and}
\item{ for every point $p\in \Gamma$, the vector space $T_pX\cap \langle \Gamma\rangle$ is trivial.}
\end{enumerate}
\end{definition}

Dually, independence of a set of points can be seen from the linear series given by the forms in $R_1$, which vanish at the points of $\Gamma$.
\begin{lemma}\label{DualIndependence}
A finite set of nonsingular points $\Gamma\subseteq X(\RR)$ is independent if and only if the linear series $V\subseteq R_1$ consisting of forms vanishing at all points of $\Gamma$ satisfies the following two properties:
\begin{enumerate}
\item{ $V$ has no common zeroes in $X$ except the points of $\Gamma$.}
\item{ For every point $p\in \Gamma$, the common kernel of the differentials at $p$ of the sections of $V$ is trivial.}
\end{enumerate}

\end{lemma}
\begin{proof} The equivalence of the first condition in both definitions is immediate. For the second recall that if $p$ is a point of $X$, then the Zariski cotangent space of $X$ at $p$ is $\operatorname{Cot}_p X:=m_p/m_p^2$, where $m_p$ is the maximal ideal of the local ring $\Osh_{X,p}$, and the tangent space of $X$ at $p$ is ${\rm Spec}(\Sym^{\bullet}(m_p/m_p^2))$. A linear form $s$ is locally specified by a function $f$ and the differential of $f$ at $p$ is $df_{p}:=(f-s(p))+m_p^2=f+m_p^2\in \operatorname{Cot}_pX$. The kernel of this differential is a vector subspace ${\rm Ker}(df_p)\subseteq T_pX$, which is well-defined independently of the representative $f$ chosen for $s$. It is immediate that $T_pX\cap\langle \Gamma\rangle=\bigcap_{f\in V_1}{\rm Ker}(df_p)$ and the equivalence of the second condition follows.     
\end{proof}

Recall that the degree of a non-degenerate projective variety is at least $\codim(X)+1$ and that a variety is called of minimal degree if this lower bound is achieved (see~\cite{EisenbudHarris} for a classification of varieties of minimal degree).

\begin{lemma}\label{MaxGI} Assume that $X\subseteq \PP^m$ is non-degenerate. The cardinality of a generic set of points, which is also  independent, is at most:
\begin{enumerate}
\item{${\rm codim}(X)+1$ if $X$ is of minimal degree, and}
\item{${\rm codim}(X)$ if $X$ is not of minimal degree.}
\end{enumerate}
Moreover, in both cases, there exist generic independent sets $\Gamma\subseteq X(\RR)$ of this maximum cardinality. In particular, a generic set of points of cardinality at most $\codim(X)$ is independent.
\end{lemma} 
\begin{proof} If $\Gamma$ is a generic set of $k$ points, then the linear series $V\subseteq R_1$ consisting of forms vanishing on $\Gamma$ has dimension $m+1-k$. Since $\Gamma$ is independent, these linear forms must have finitely many zeroes on $X$ and, in particular, $\dim(X)-(m+1-k)\leq 0$ and thus $k\leq \codim(X)+1$. If $k=\codim(X)+1$, then $V$ must have $\deg(X)$ many zeroes on $X$, which, by independence of $\Gamma$, can happen only if $\deg(X)=k$ (i.e., if $X$ is a variety of minimal degree). Since $X$ is totally real, this bound is achieved by a generic set of $\codim(X)+1$ real points. 
As a result, if $X$ is not a variety of minimal degree, then we know that $k\leq {\rm codim}(X)$ and we will show that this upper bound can be achieved by a set $\Gamma\subseteq X(\RR)$. We begin by showing that there exist linear forms $L_1,\dots, L_{d}$ with $d=\dim(X)$ such that $X\cap \mathbb{V}(L_1,\dots, L_d)$ is a set of reduced points in linearly general position of which at least $\codim(X)+1$ are real. To see this recall the following facts:
\begin{enumerate}
\item{The intersection of a non-degenerate, positive dimensional, variety with a general hyperplane $H$ is non-degenerate in $H$ (see~\cite[Proposition 18.10]{Harris}).}
\item{By Bertini's Theorem~\cite[Theorem 6.3]{J}, the general hyperplane section of a geometrically integral variety of dimension at least two is geometrically integral and the general hyperplane section of a (reduced) variety is reduced.}
\item{A real variety is totally real iff it contains a smooth real point (see~\cite[Section 1]{B}).}
\item{The locus of hyperplanes, which intersect the nonsingular locus of $X$ transversely, contains the Zariski open set given by the complement of the projective dual of $X$ (see ~\cite[Chapter 1]{GKZ}).}
\end{enumerate}
As a result, since $X$ is totally real, there exists a form $L_1\in R_1$ defining a hyperplane $H=\mathbb{V}(L_1)$, which intersects the nonsingular locus of $X$ transversely, passes through some real nonsingular point $p$, and has the property that $H\cap X$ is non-degenerate and geometrically integral, if $X$ has dimension at least two. Since $H$ is transverse, it follows that $X\cap H$ is again totally real. Iterating this process we can construct linear forms $L_1,\dots, L_{d-1}$ such that $X\cap \mathbb{V}(L_1,\dots, L_{d-1})$ is a totally real, non-degenerate curve $C\subseteq \PP^{m-d+1}$. The hyperplane $H'$ of $\PP^{m-d+1}$ spanned by a general set of $m-d+1$ real points of $C$ contains $\deg(C)$ reduced points of which at least the chosen $m-d+1=\codim(X)+1$ points are real. By the genericity of our choices this hyperplane is also transverse to the curve at every point of intersection. Moreover, by the uniform position Lemma of Harris~\cite{H}, these $\deg(C)$ points are in linearly general position in $\PP^{m-d}$. Let $\Gamma$ be any subset of size $\codim(X)$ of these real points. Since the points are in linearly general position, the subspace $\langle\Gamma\rangle$ of $H'$ spanned by the points of $\Gamma$ in $H'$ contains no other point of $X$, and the tangent space of $C$ is transverse to $H'$ at every point of $C$. It follows that $\Gamma$ is an independent set of points of $X$ as claimed.  
\end{proof}

\begin{remark} The maximum cardinality of an independent set of points and the maximum cardinality of a {\it generic} independent set of points may differ. For an example, let $X\subseteq \mathbb P^3$ be a nonsingular cubic surface. A set of three distinct collinear points of $X$, which span a line in $\PP^3$ not contained in the surface, gives an example of a maximum independent set of points (any set of larger cardinality in $X$ would have to span at least a plane, which must intersect the surface along at least a curve). On the other hand, the maximal, generic, independent set of points of $X$ has cardinality one, since the line spanned by a generic set of two points intersects the surface at a third point not in the set.
\end{remark}

Our next theorem explains why the concept of independence is useful for determining the dimension of $P(\Gamma)$. It generalizes~\cite[Proposition 1.3]{BIK} from the context of Veronese embeddings of projective space to arbitrary projective varieties. Recall that if $p\in X$ is a (closed) point, then a function $f\in \mathcal{O}_{X,p}$ (resp. a section $s$) is said to vanish to order at least two if $f$ (resp. some representative $f$ for $s$ near $p$) is an element of $m_p^2$, where $m_p$ is the maximal ideal of $\Osh_{X,p}$.

\begin{thm} \label{DimP1}Let $\Gamma\subseteq X(\RR)$ be a finite set of nonsingular points. If $\Gamma$ is independent, then the cone $P(\Gamma)$ has dimension ${\rm dim}(B)$, where $B\subseteq R_2$ is the vector space of quadratic forms vanishing to order at least two at all points of $\Gamma$.
\end{thm}
\begin{proof} Let $p$ be a nonsingular point of $X$. Every zero of a nonnegative quadratic form at $p$ is a critical point of the form and thus $P(\Gamma)\subseteq B$. We will show that the independence of $\Gamma$ implies that $P(\Gamma)$ is full-dimensional in $B$. Let $V$ be the set of linear forms vanishing at all points of $\Gamma$. Let $s_1,\dots,s_k$ be a basis for $V$ and define $r:=\sum_{i=1}^k s_i^2$. Since $\Gamma$ is independent, it follows that the only real zeroes of $r$ are the points of $\Gamma$ and that the Hessian of the form $r$ is strictly positive definite at every $p\in \Gamma$. As a result, if $t\in B$, then for any sufficiently small real number $\epsilon>0$ the section $r+\epsilon t$ is nonnegative. It follows that the cone $P(\Gamma)$ is full-dimensional in $B$ as claimed.\end{proof}

What is the dimension of the vector space $B$? If the points of $\Gamma$ are generic, then Terracini's Lemma shows that this dimension is determined by the codimension of a certain secant variety. The fact that these secant varieties are never deficient allows us to explicitly compute the dimension of $B$ and prove Theorem~\ref{thm:DimP}. 
\begin{proof}[Proof of Theorem~\ref{thm:DimP}] By Theorem~\ref{DimP1}, the independence of $\Gamma$ implies that the dimension of $P(\Gamma)$ is equal to that of the linear system of quadrics vanishing at the points of $\Gamma$ with multiplicity at least two. To compute the dimension of this linear system, let $\nu$ be the second Veronese morphism $\nu :X\rightarrow \PP(R_2^*)$. A quadric in $R_2$ vanishes at a point $p$ of $\Gamma$ with multiplicity at least two iff the corresponding linear form in $\PP(R_2^*)$ contains the tangent space $T_{\nu(p)}Z$, where  $Z\subseteq \PP(R_2^*)$ denotes the image of $X$ under the second Veronese map. Now, by Terracini's Lemma~\cite{Terracini}, the projective subspace of $\PP(R_2^*)$ spanned by the tangent spaces $T_{\nu(p_i)}Z$ coincides with the tangent space of the $(s-1)$\textsuperscript{st} secant variety of $Z$ at a general point in the projective subspace spanned by the points $\nu(p_i)$, where $\Gamma=\{p_1,\ldots,p_s\}$. It follows that the dimension of the space of quadrics in $R_2$ vanishing at the points of $\Gamma$ with multiplicity at least two equals the dimension of the space of linear forms in $\PP(R_2^*)$, which contain the tangent space of ${\rm Sec}^{\codim(X)-1}(Z)$ at a general point. Since the points $p_i$ are generic, this dimension equals $\codim({\rm Sec}^{s-1}(Z))$ as claimed.
On the other hand, the $(s-1)$\textsuperscript{st} secant variety of $Z$ is the image of the rational map $Z^{s}\times \PP^{s-1}$, which maps $(z_1,\dots, z_s,[\alpha_1:\dots: \alpha_{s}])$ to $\alpha_1z_1+\dots + \alpha_s z_s$ and thus the dimension of the secant variety is at most $s(\dim(Z)+1)-1$. The $(s-1)$\textsuperscript{st} secant variety of $Z$ is called non-defective if this obvious upper bound is attained. By \cite[Theorem 3.3]{Zak}, if $s\leq \codim(X)+1$ then ${\rm Sec}^{s-1}(Z)$ is non-defective. By Lemma \ref{MaxGI}, we know that $s\leq \codim(X)+1$, proving the claim.
\end{proof}

\section{Dimensional differences and gap vectors}
\label{DGepsilon}
In this section we prove our main result, Theorem~\ref{DimGap}, which computes the dimensional gap between $\Sigma(\Gamma)$ and $P(\Gamma)$ for generic sets of points, in terms of the quadratic deficiency of $X$ and that of its generic linear projections. We introduce the gap vector of a real projective varietiy and study its properties.

\begin{definition} Let $X\subseteq \PP^m$ be a non-degenerate variety of dimension $d$ and codimension $c$. Let $I$ be the ideal of definition of $X$, and let $R$ be its homogeneous coordinate ring. Recall~(\cite{BSV}, \cite{Zak}) that the quadratic deficiency of $X$ is the number
\[ \epsilon(X):= \binom{c+1}{2}-\dim(I_2)=\dim(R_2)-(m+1)(d+1)+\binom{d+1}{2}.\]
\end{definition}
We are now in a position to prove Theorem \ref{DimGap}:

\begin{proof}[Proof of Theorem \ref{DimGap}] 

Using the definition of quadratic deficiency of $X$ and $Y$ we obtain the equalities
\[ \dim(P(\Gamma))= \epsilon(X) +(m+1-s)(d+1)-\binom{d+1}{2} \]
\[ \dim(\Sigma(\Gamma))= \epsilon(Y) +(m-s+1)(d+1)-\binom{d+1}{2}.\]
Subtracting we conclude that the equality $\dim(P(\Gamma))-\dim(\Sigma(\Gamma))= \epsilon(X)-\epsilon(Y)$ holds as claimed.
\end{proof}

\subsection{Combinatorics and geometry of gap vectors}\label{GapVectors}
As described in the introduction we keep track of the dimensional gaps between generic faces of $P$ and $\Sigma$ by putting them into a vector.
\begin{definition} The gap vector of $X$ is the vector $g(X)\in \mathbb{Z}^{\codim(X)}$ defined for $1\leq j\leq \codim(X)$, by $g(X)_j=\dim (P(\Gamma))-\dim(\Sigma(\Gamma))$, where $\Gamma\subseteq X(\RR)$ is a generic independent set with $|\Gamma|=j.$\end{definition}

In this section we prove several fundamental properties of gap vectors.  Our results are a consequence of Theorem~~\ref{DimGap} and use the following elementary geometric construction,

\begin{definition}Let $W\subseteq \PP^m$ be a variety. A point $p\in \PP^m$ is a cone point for $W$ if $p\in W$ and for any $q\in W$, the line $\langle p,q\rangle$ is contained in $W$.
\end{definition}
\begin{lemma}\label{ConePoints} The following statements hold,
\begin{enumerate}
\item{The set of cone points of $W$ forms a projective subspace of $\PP^m$ contained in $W$.}
\item{Let $p\in \PP^m$ and let $W'$ be the image of $W$ under the projection $\pi_p$ away from $p$. The point $p$ is a cone point of $W$ iff the equality $\pi_p^*\left(I(W')\right)=I(W)$ holds. }
\end{enumerate}
\end{lemma}

We are now in a position to prove the main theorem of this section,
\PropGap*
\begin{proof} $(1)$ Since $\Sigma(\Gamma)\subseteq P(\Gamma)$, the numbers $g_j(X)$ are nonnegative for $1\leq j\leq c$. This also implies that if $Y$ is obtained from $X$ by projecting away from a generic set of $s$ points of $X$, then $\epsilon(X)\geq \epsilon(Y)$ because $\epsilon(X)-\epsilon(Y)=g_s(X)$. It follows that $g(X)$ is non-decreasing, because $g_{j+1}(X)-g_j(X)=\epsilon(Y_{j})-\epsilon(Y_{j+1})$, where $Y_{j+1}$ is obtained from $Y_{j}$ by projecting away from a general point in $Y_{j}.$

$(2)$ If $|\Gamma|=c$, then the projection away from $\Gamma$, $\pi_{\Gamma}: X\dashrightarrow \PP^{\dim(X)}$ has image $Y=\PP^{\dim(X)}$ and thus $\epsilon(Y)=0$, proving the claim. 

$(3)$ The image of $X$ under the projection away from $c-1$ general real points is a hypersurface $Y$ of degree $\deg(X)-c+1$ in $\PP^{d+1}$ because both degree and codimension decrease by one when we project away from a general point on a non-degenerate variety of codimension at least two. If $X$ is not of minimal degree, then the degree of this hypersurface is at least three and a direct calculation shows that $\epsilon(Y)=1$, proving the claim.

$(4)$ From the definition of quadratic deficiency we obtain that for all $j$ satisfying $1\leq j\leq c-1$ the following equalities hold
\[ g_{j+1}(X)-g_j(X)=\epsilon(Y_{j})-\epsilon(Y_{j+1})= \codim(Y_j) + \dim\left(I(Y_{j+1})_2\right)-\dim\left(I(Y_j)_2\right).\]
The right hand side is at most $\codim(Y_j)=c-j$; if $\pi_p$ denotes the projection away from the $(j+1)$\textsuperscript{st} general point, then $\pi_p^*$ injects $I(Y_{j+1})_2$ into $I(Y_j)_2$. Moreover, the inequality is an equality precisely when the dimensions of $I(Y_j)_2$ and $I(Y_{j+1})_2$ coincide.  This happens when $I(Y_{j})$ does not contain any quadric and we claim that this is the only possibility. If the variety $Z$ defined by the quadrics in $I(Y_j)$ is not $\PP^m$, then by Lemma~\ref{ConePoints}  the projective subspace of cone points of $Z$ is a proper subspace of $\PP^m$, which cannot contain the non-degenerate variety $Y_j$. It follows that a generic point of $Y_j$ is not a cone point of $Z$ and thus that $\pi_p^*\left(I(Y_{j+1})_2\right)\neq I(Y_j)_2$ implying that $g_{j+1}(X)-g_j(X)<c-j$. This concludes the proof of the theorem.
\end{proof}

The basic properties above yield the following classification theorem for the varieties with simplest gap vectors,

\PropGapVect*
\begin{proof} 
$(1)$: By~\cite[Corollary 5.8]{Zak}, we have $\epsilon(X)=0$ iff $X$ is a variety of minimal degree. Thus $(1)$ follows from Theorem~\ref{gapVector} part $(2)$.

$(2)$: By part $(1)$, if $g(X)$ has some nonzero component, then $X$ is not of minimal degree and in particular the last component of $g(X)$ has to be nonzero. If this last component is greater than one, then $g(X)_{\codim(X)-1}$ is also nonzero. As a result if $g(X)$ has only one nonzero component, then $\epsilon(X)=1$. By Theorem~\ref{gapVector} part $(3)$, $g_{\codim(X)}(X)=1$ implies that all other components are zero as claimed. \end{proof}

\begin{remark} By~\cite[Corollary 5.8]{Zak} we have $\epsilon(X)=0$ iff $X$ is a variety of minimal degree. As a result, the previous corollary implies that if $X$ is not a variety of minimal degree and $\Gamma$ is a maximal generic and independent set of points, then $\Sigma(\Gamma)\subseteq P(\Gamma)$ is a strict inclusion between faces of different dimension (and in particular $\Sigma\neq P$). This gives a new proof of the theorem of Blekherman, Smith and Velasco~\cite{BSV}, which states that $\Sigma\neq P$ for varieties $X$, which are not of minimal degree.
\end{remark} 
\begin{remark}By ~\cite[Proposition 5.10]{Zak} we know that $\epsilon(X)=1$ iff either:
\begin{enumerate}
\item{$X$ is a hypersurface of degree at least three.}
\item{ $X$ is a linearly normal variety of almost-minimal degree (i.e., $X$ satisfies $\deg(X)=\codim(X)+2$).}
\end{enumerate}
Moreover~\cite[Theorem 1.2]{BSc} shows that a non-degenerate variety $X$ of almost minimal degree is linearly normal iff it is arithmetically Gorenstein. These varieties are also known as maximally Del Pezzo varieties and have been classified by Fujita~\cite{Fujita}. The simplest examples are Del Pezzo surfaces~\cite{Pezzo} of degree at least three embedded in projective space by their anticanonical linear system. These are the nonsingular surfaces of degree $m$ in $\PP^m$ for $3\leq m \leq 9$. 
\end{remark}
\begin{remark}
Observe that equality of dimensions does not imply equality of faces. The first component of the gap vector is always equal to $0$. However, if $X$ is not a variety of minimal degree, $\Sigma$ is different from $P$ and it can never happen that $\Sigma(p)=P(p)$ for all points $p\in X$. On the other hand, if the dimensions of all faces $\Sigma(\Gamma)$ and $P(\Gamma)$ coincide, then $X$ is a variety of minimal degree and $\Sigma=P$.
\end{remark}

Part $(2)$ of Corollary \ref{cor:GapProp} is a quantitative counterpart of the idea that sums of squares of linear forms and nonnegative quadratic forms are closest, but unequal, on hypersurfaces of degree at least three and on linearly normal varieties of almost minimal degree.\\
Finally, we compute the gap vector of the Veronese embeddings of $\PP^2$. More intrinsically, the following theorem determines the dimensional differences between nonnegative polynomials of degree $2d$ and sums of squares of degree $d$ in three variables vanishing at generic sets of points of size $j$ with $1\leq j\leq \binom{d+2}{2}-3$. 

\begin{thm}\label{ThmPP2} Let $d$ be a nonnegative integer and let $X:=\nu_d(\PP^2)\subseteq \PP^{\binom{d+2}{2}-1}$ be the $d$\textsuperscript{th} Veronese embedding of $\PP^2$. For $1\leq j\leq c:=\binom{d+2}{2}-3$ the components of the gap vector of $X$ are given by the formula 
\[
g_j(X)=\begin{cases}
0,&\text{ if $j\leq \binom{d+1}{2}$,}\\
(j-\binom{d+1}{2})(d-1)-\binom{j+1-\binom{d+1}{2}}{2},&\text{ otherwise.}
\end{cases}
\] 
\end{thm}
\begin{proof} By \cite[Corollary 1.8]{BIK}, there is no dimensional gap when $j\leq \binom{d+1}{2}$.
We know that
\[\sum_{i=2}^{\codim(X)} (g_i-g_{i-1})=\epsilon(X)-g_1(X)=\epsilon(X) \]
By direct computation $\epsilon(X)=\binom{d-1}{2}$. Moreover, by Theorem~\ref{gapVector} part $(4)$, we know that
\begin{equation}\label{eq:gap}
 g_{\binom{d+1}{2}+s+1}(X)-g_{\binom{d+1}{2}+s}(X)\leq d-2-s.
\end{equation}
As a result the sum of the entries in the first difference of the gap vector cannot exceed $\sum_{s=1}^{c-\binom{d+2}{2}} (d-2-s)=\binom{d-1}{2}$. Since the sum of the entries in the first difference of the gap vector also equals $\epsilon(X)=\binom{d-1}{2}$, the inequality \eqref{eq:gap} is an equality and thus the gap vector of $X$ has maximum growth in all of its components $g_j(X)$ for $j\geq \binom{d+1}{2}$. The claim follows immediately.
\end{proof}

\begin{remark}
Although the previous theorem shows that the gap vectors of Veronese embeddings of $\PP^2$ exhibit only extremal growth, this is not necessarily true for arbitrary varieties. For instance, for the $4$\textsuperscript{th} Veronese embedding of $\PP^3$, using Macaulay2~\cite{M2}, we compute that the first $23$ entries of the gap vector are equal to $0$ and $g_{24},\ldots,g_{31}$ are $3,10,16,21,25,28,30,31$. Whereas, assuming maximal growth, the first nonzero entry should be an $8$.
\end{remark}

For $X = v_d(\mathbb P^n)$, in \cite[Proposition 6.2]{BIK}, bounds were established for the smallest $j$ such that $g_j(X) > 0$. We state the following conjecture, which implies that these bounds are in fact sharp.
 
\begin{conj}
\label{conjecture}
 Let $X = v_d(\mathbb P^n)$ and $\bar j := \min\{j: g_j(X) > 0\}$. Then the following hold.
\begin{enumerate}
 \item $\bar j = \left\lceil\binom{n+d}{d}-(n+1)+\frac{1}{2}-\sqrt{\left(n+\frac{1}{2}\right)^2+2\binom{n+2d}{2d}-2(n+1)\binom{n+d}{d}}\right\rceil,$
\item for $\bar j \leq j \leq \codim X$ it holds that 
$$g_j(X) = \binom{n+2d}{2d} - j(n+1) - \binom{\binom{n+d}{d}-j+1}{2}.$$
\end{enumerate}
\end{conj}

\begin{remark}
In the proof of Theorem \ref{ThmPP2}, a key step was to establish the fact that the differences $g_j - g_{j-1}$ of the gap vector sum up to the quadratic deficiency $\varepsilon(X)$. We have verified experimentally for $n \in \{3,4\}$ and $2 \leq d\leq 100$ as well as for $2d = 4$ and $4\leq n\leq 100$ that the differences given in Conjecture \ref{conjecture} also sum up to $\varepsilon(X) = \binom{n+2d}{2d} - (n+1)\binom{n+d}{d} + \binom{n+1}{2}$.
\end{remark}

\end{document}